\documentclass[final]{siamltex}

\usepackage{amsfonts,amsmath,amssymb}
\usepackage{mathrsfs,mathtools}
\usepackage{enumerate}
\usepackage{hyperref}
\usepackage{esint}
\usepackage{graphicx}
\usepackage{bm}
\usepackage{cite}
\usepackage{color}
\usepackage{stackengine}
\usepackage{subcaption}

\DeclareMathAlphabet{\mathpzc}{OT1}{pzc}{m}{it}

\newcommand{\dd}{\mathrm{d}}
\newcommand{\oo}{\mathrm{o}}
\newcommand{\ii}{\mathrm{i}}
\newcommand{\ff}{\mathrm{f}}

\newtheorem{remark}{Remark}

\title{Auxiliary functions in the study of Stefan-like problems with variable thermal properties 
}
\author{Andrea N. Ceretani \thanks{IMAS-UBA-CONICET, Intendente Guiraldes 2160, Capital Federal 1428, Argentina, \texttt{aceretani@dm.uba.ar}, and Escuela de Ciencia y Tecnolog\'ia, Universidad Nacional de San Mart\'in, Mart\'in de Irigoyen 3100, San Mart\'in 1650, Argentina, \texttt{aceretani@unsam.edu.ar}}
\and
Natalia N. Salva  \thanks{CNEA-CONICET, Departamento de Mec\'anica Computacional, Centro At\'omico Bariloche, Av. Bustillo 9500, Bariloche 8400, Argentina, and Departamento de Matem\'atica, Centro Regional Bariloche, Universidad Nacional del Comahue, Quintral 1250, Bariloche 8400, Argentina, \texttt{{natalia@cab.cnea.gov.ar}}}
\and
Domingo A. Tarzia \thanks{CONICET-Departamento de Matem\'atica, Facultad de Ciencias Empresariales, Universidad Austral, Paraguay 1950, Rosario 2000, Argentina, \texttt{dtarzia@austral.edu.ar}}}

\begin{document}  

\maketitle

\begin{abstract}
We address the existence and uniqueness of the so-called {\em modified error function} that arises in the study of phase-change problems with specific heat and thermal conductivity given by linear functions of the material temperature. This function is defined from a differential problem that depends on two parameters which are closely related with the slopes of the specific heat and the thermal conductivity. We identify conditions on these parameters which allow us to prove the existence of the modified error function. In addition, we show its uniqueness in the space of non-negative bounded analytic functions for parameters that can be negative and different from each other. This extends known results from the literature and enlarges the class of associated phase-change problems for which exact similarity solutions can be obtained. In addition, we provide some properties of the modified error function considered here.
\end{abstract}

\begin{keywords}
Modified error function, Stefan problems, temperature-dependent thermal coefficients, nonlinear boundary value problems
\end{keywords}

\begin{AMS}
34B15, 
34B08, 
80A22, 
35R35 
\end{AMS}

\section{Introduction}
This article is devoted to show the existence and uniqueness of an auxiliary function that arise in the study of phase-change processes when some thermal coefficients are assumed to vary with the material temperature. The function of interest is the solution to the following nonlinear differential system:
\begin{equation}\label{P}
((1+\delta y)y')'+2x(1+\gamma y)y'=0\quad x\in(0,+\infty),\quad y(0)=0,\quad y(+\infty)=1,
\end{equation}
where $\delta,\gamma\in(-1,+\infty)$, and $y(+\infty):=\lim_{x\to+\infty}y(x)$.

Problem \eqref{P} was introduced by Oliver and Sunderland in \cite{OlSu1987} for the study of two-phase Stefan problems on the semi-infinite line $(0,+\infty)$, with heat capacity and thermal conductivity given by
\begin{equation}\label{kc}
c(\vartheta):=c_\oo\left(1+\alpha\frac{\vartheta-\vartheta_\oo}{\vartheta_\ii-\vartheta_\oo}\right)\qquad\text{and}\qquad
k(\vartheta):=k_\oo\left(1+\beta\frac{\vartheta-\vartheta_\oo}{\vartheta_\ii-\vartheta_\oo}\right),
\end{equation}
respectively, where $\vartheta$ is the material temperature, $c_\oo>0$ and $k_\oo>0$ are reference values for the specific heat and the thermal conductivity, $\vartheta_\ii$ is a uniform initial temperature distribution, $\vartheta_\oo$ is a prescribed constant temperature at the boundary $x=0$, and $\alpha,\beta\in\mathbb{R}$ (the values of $\alpha$, $\beta$, $c_\oo$, and $k_\oo$ generally differ from liquid to solid phases).
In particular, they proposed a method to find similarity solutions that relies on the assumption that problem \eqref{P} has a solution $\Phi_{\delta\gamma}$ for any $\delta,\gamma\in(-1,+\infty)$. The temperature in each phase is then obtained in terms of an auxiliary function $\Phi_{\delta\gamma}$ for some parameters $\delta$ and $\gamma$ that must be calculated and may be different from one phase to the other, see equations (23)-(27) in \cite{OlSu1987}. The function $\Phi_{\delta\gamma}$ was called {\em modified error function} and it was obtained numerically by 
solving problem \eqref{P}. 

The signs of $\delta$ and $\gamma$ are closely related on how the thermal conductivity and the heat capacity vary with the material temperature. We illustrate this in the following example: Consider a one-phase solidification process for a material with phase-change temperature $\vartheta_\ff$. Then, $\vartheta_\oo<\vartheta_\ff$ and $\vartheta_\ii\equiv\vartheta_\ff$. According with \cite{OlSu1987}, it must be $\gamma\Phi_{\delta\gamma}(\lambda)=\alpha$ and $\delta\Phi_{\delta\gamma}(\lambda)=\beta$,
where $\lambda\equiv\frac{s(t)}{2\sqrt{at}}$ for all $t>0$, $x=s(t)$ is the location of the free boundary at time $t$, and $a$ is the coefficient of diffusion of the solid phase. Then, provided that $\alpha\neq 0$ and $\beta\neq 0$ we have $\mathrm{sign}(\gamma)=\mathrm{sign}(\alpha)$ and $\mathrm{sign}(\delta)=\mathrm{sign}(\beta)$ since the modified error function is non-negative everywhere (see Theorem \ref{theorem} below). We now observe that the slopes of $c$ and $k$ are given by $\frac{c_\oo\alpha}{\vartheta_\ii-\vartheta_\oo}$ and $\frac{k_\oo\beta}{\vartheta_\ii-\vartheta_\oo}$, respectively, see \eqref{kc}, where $\vartheta_\ii-\vartheta_\oo>0$. Therefore, the heat capacity is increasing for $\gamma>0$, decreasing for $\gamma<0$, and the analogous conclusion holds for the relation between the thermal conductivity and the parameter $\delta$.

The method described in \cite{OlSu1987} had already been considered by Cho and Sunderland in \cite{ChSu1974} for the analogous Stefan problem with constant heat capacity, corresponding to $\gamma=0$ in \eqref{P}. The modified error function for the case when $\gamma=0$ and $\delta>0$ was studied by the authors in \cite{CeSaTa2017,CeSaTa2018} (see also \cite{Bo2018}), where existence and uniqueness in the space of bounded analytic functions were proven and explicit approximations were provided (see \cite{MaSiPa2019} for improved approximations). In particular, this paper extends the existence and uniqueness result in \cite{CeSaTa2017} for the case $\delta<0$ ($\gamma=0$). 

Similar approaches to those introduced by Sunderland and collaborators were followed to find exact similarity solutions in the cases when non-Dirichlet boundary conditions are prescribed at $x=0$ or when the physical domain is allowed to move, see e.g. \cite{CeSaTa2018-b,KuKuRa2018-a,KuKuRa2018-b}. In all cases it was assumed that $\alpha=\beta>0$, or $\alpha=0$ and $\beta>0$. Analogous methods were also used to determine solutions to problems with more general thermal coefficients, see e.g. \cite{BoNaSeTa2020,KuKuRa2018-c}. Other approaches to find similarity solutions to Stefan-like problems with non-constant thermal properties and arbitrary initial and boundary conditions were recently considered in, e.g., \cite{BrNa2019,Fo2018,ZhouWaBu2014,BoTa2018,Tao1978}. We refer to \cite{OlSu1987} for a discussion about the effects of including variations with respect to the material temperature of thermal conductivity and heat capacity in phase-change models.

In the next section we provide the existence and uniqueness of a solution $\Phi_{\delta\gamma}$ to problem \eqref{P} in the space of bounded analytic functions, for some $\delta,\gamma\in(-1,+\infty)$. In particular, $\delta$ and $\gamma$ are allowed to be negative and different from each other. In this manner we extend already known results for modified error functions available in the literature. 

\section{Existence and uniqueness of $\Phi_{\delta\gamma}$}
The following is the main result of the paper. 
\begin{theorem}\label{theorem}
Assume that  
\begin{equation}\label{Condition-delta-gamma}
M(\delta,\gamma):= \frac{\max(1,1+\delta)^{3/2}\max(1,1+\gamma)^{1/2}}{\min(1,1+\delta)^{5/2}\min(1,1+\gamma)^{1/2}} 
\left(\hspace*{-0.05cm}2|\delta|+\frac{|\delta-\gamma|\max(1,1+\delta)}{\min(1,1+\delta)\min(1,1+\gamma)}\hspace*{-0.05cm}\right)<1.
\end{equation}
Then problem \eqref{P} admits a unique bounded analytic solution $\Phi_{\delta\gamma}$ that satisfies 
\begin{equation*}
0\leq\Phi_{\delta\gamma}(x)\leq 1\qquad\text{for all}\quad x\geq 0.
\end{equation*}
\end{theorem}

\begin{proof}
Let $X$ be the Banach space of bounded analytic functions $h:[0,+\infty)\to\mathbb{R}$, equipped with the supremum norm $\|h\|_\infty:=\sup \{|h(x)|:x\geq 0\}$. In addition, let
\begin{equation*}
K:=\{h\in X:\|h\|_\infty\leq 1, h\geq 0, h(0)=0\}.
\end{equation*}  
We notice that $K$ is a non-empty closed subset of $X$.

Let $h\in K$. Consider the following auxiliary linear problem:
\begin{equation}\label{AuxP}
((1+\delta h)y')'+2x(1+\gamma h)y'=0\quad x\in(0,+\infty),\quad y(0)=0,\quad y(+\infty)=1.
\end{equation}
From this, we shall formulate the existence and uniqueness of a solution to \eqref{P} as a fixed point problem. 

Let $F(\,\cdot\,;h):[0,+\infty)\to\mathbb{R}$ be given by
\begin{equation*}\label{F}
F(x;h):=\int_0^x\exp\left(-\int_0^w\frac{2z(1+\gamma h(z))}{1+\delta h(z)}\,\dd z\right)\frac{1}{1+\delta h(w)}\,\dd w.
\end{equation*}
Exploiting the inequalities
\begin{equation}\label{Estimate-h}
0<\min(1,1+\upsilon)\leq 1+\upsilon h\leq \max(1,1+\upsilon),\quad\text{for}\quad \upsilon=\delta\quad\text{or}\quad\upsilon=\gamma,
\end{equation}
and taking into account that $\int_0^{+\infty}\exp(-w^2)\,\dd w=\frac{\sqrt{\pi}}{2}$, we obtain 
\begin{equation*}
\begin{split}
0\leq F(x;h)\leq\,\frac{\sqrt{\pi}\max(1,1+\delta)^{1/2}}{2\min(1,1+\delta)\min(1,1+\gamma)^{1/2}}=:M_1(\delta,\gamma)\qquad\text{for all}\quad x\geq 0.
\end{split}
\end{equation*}
Then $F(\,\cdot\,;h)$ is well-defined. Analogous computations allow us to observe 
\begin{equation}\label{Estimate-FInf}
\frac{1}{F(+\infty;h)}\leq \frac{2\max(1,1+\delta)\max(1,1+\gamma)^{1/2}}{\sqrt{\pi}\min(1,1+\delta)^{1/2}}=:M_2(\delta,\gamma).
\end{equation}
Hence, the map $T:K\to K$ given by
\begin{equation}\label{T}
(T(h))(x):=\frac{F(x;h)}{F(+\infty;h)},
\end{equation}
is well-defined too. Then, we notice that a necessary and sufficient condition to be $y\in K$ a solution to \eqref{P} is that $T(y)=y$ since $y=T(h)$ solves \eqref{AuxP} for each $h\in K$. The rest of the proof consists of proving that $T$ is a contraction from $K$ into itself, provided \eqref{Condition-delta-gamma} holds true. Then, the theorem will follow by Banach's fixed point theorem. 

Initially, we note $T(K)\subset K$ as a direct consequence of the definition of $T$. Let $x\geq 0$ and $h_1,h_2\in K$. We have:
\begin{equation*}
\begin{split}
|(T(h_1))(x)-(T(h_2))(x)|\leq\,&\left|\frac{F(x;h_1)}{F(+\infty;h_1)}-\frac{F(x;h_2)}{F(+\infty;h_1)}\right|+
\left|\frac{F(x;h_2)}{F(+\infty;h_1)}-\frac{F(x;h_2)}{F(+\infty;h_2)}\right|\\[0.25cm]
\leq\,& \frac{|F(x;h_1)-F(x;h_2)|}{F(+\infty;h_1)}+\frac{\left|F(+\infty;h_2)-F(+\infty;h_1)\right|}{F(+\infty;h_1)}.
\end{split}
\end{equation*}
Then, using \eqref{Estimate-FInf} we find
\begin{equation}\label{Estimate-T}
\begin{split}
|(T(h_1))(x)-(T(h_2))(x)|
\leq M_2(\delta,\gamma)(I+J),
\end{split}
\end{equation}
where $I:=|F(x;h_1)-F(x;h_2)|$ and $J:=\,|F(+\infty;h_1)-F(+\infty;h_2)|$.

Defining $f(w):=\exp(-2w)$ and $w_i:=\int_0^w\frac{z(1+\gamma h_i(z))}{1+\delta h_i(z)}\,\dd z$ for $w\geq 0$, $i=1,2$, and using the estimates \eqref{Estimate-h}, we have
\begin{equation}\label{Estimate-I}
\begin{split}
I\leq\,&\left| \int_0^x\frac{f(w_1)}{1+\delta h_1(w)}-\frac{f(w_2)}{1+\delta h_1(w)} \,\dd w \right|+\left| \int_0^x\frac{f(w_2)}{1+\delta h_1(w)}-\frac{f(w_2)}{1+\delta h_2(w)} \,\dd w \right|\\[0.25cm]
\leq\,&\frac{1}{\min(1,1+\delta)}\int_0^x\left|f(w_1)-f(w_2)\right|\,\dd w+ \frac{|\delta|\|h_1-h_2\|_\infty}{\min(1,1+\delta)^2}\int_0^xf(w_2)\,\dd w.
\end{split}
\end{equation}
Let $w\geq 0$. From the Mean Value Theorem we observe that $f(w_1)-f(w_2)=-2f(\omega)(w_1-w_2)$, where $\omega$ is a number in between $w_1$ and $w_2$. Furthermore, since $f$ is decreasing we have that $f(\omega)\leq f(\min(w_1,w_2))=f(w_k)$ for $k=1$ or $k=2$. Then, 
\begin{equation}\label{Estimate-I1}
\begin{split}
\int_0^x|f(w_1)-f(w_2)|\,\dd w\leq\,&\,2\int_0^xf(w_k)|w_1-w_2|\,\dd w\\[0.25cm]
\leq\,&\frac{|\delta-\gamma|\|h_1-h_2\|_\infty}{\min(1,1+\delta)^2}\int_0^xw^2\exp\left(-\frac{\min(1,1+\gamma)}{\max(1,1+\delta)}w^2\right)\,\dd w\\[0.25cm]
\leq\,&\frac{\sqrt{\pi}|\delta-\gamma|\max(1,1+\delta)^{3/2}}{4\min(1,1+\delta)^2\min(1,1+\gamma)^{3/2}}\|h_1-h_2\|_\infty.
\end{split}
\end{equation}
where we have used \eqref{Estimate-h} and that $\int_0^{+\infty}w^2\exp(-w^2)\,\dd w=\frac{\sqrt{\pi}}{4}$.

From analogous arguments, we find
\begin{equation}\label{Estimate-I2}
\begin{split}
\int_0^xf(w_2)\,\dd w\leq\,&\frac{\sqrt{\pi}\max(1,1+\delta)^{1/2}}{2\min(1,1+\gamma)^{1/2}}.
\end{split}
\end{equation}

Hence, using \eqref{Estimate-I1} and \eqref{Estimate-I2} on \eqref{Estimate-I}, we obtain $I\leq M_3(\delta)\|h_1-h_2\|_\infty$, where
\begin{equation*}
\begin{split}
M_3(\delta):=\frac{\sqrt{\pi}\max(1,1+\delta)^{1/2}}{4\min(1,1+\delta)^2\min(1,1+\gamma)^{1/2}} 
\left(2|\delta|+\frac{|\delta-\gamma|\max(1,1+\delta)}{\min(1,1+\delta)\min(1,1+\gamma)}\right).
\end{split}
\end{equation*} 
An identical argument yields $J\leq M_3(\delta,\gamma)\|h_1-h_2\|_\infty$. Hence, it follows from \eqref{Estimate-T} and the above estimates for $I$ and $J$ that 
\begin{equation*}
\begin{split}
\|T(h_1)-T(h_2)\|_\infty\leq 2M_2(\delta,\gamma)M_3(\delta,\gamma)\|h_1-h_2\|_\infty=M(\delta,\gamma)\|h_1-h_2\|_\infty,
\end{split}
\end{equation*} 
Therefore, condition \eqref{Condition-delta-gamma} ensures that $T$ is a contracting map and the proof is finished.
\end{proof}

\begin{remark}
Theorem \ref{theorem} improves the analogous result given by the authors in \cite{CeSaTa2017} for the case when $\delta>0$ and $\gamma=0$. In fact, the result given in \cite{CeSaTa2017} holds true provided that $\delta>0$ satisfies $\frac{\delta}{2}(1+\delta)^{3/2}(3+\delta)(1+(1+\delta)^{3/2})<1$,
whereas condition \eqref{Condition-delta-gamma} in Theorem \ref{theorem} just requires $\delta(1+\delta)^{3/2}(3+\delta)<1$.
\end{remark}

Figure \ref{fig:A} depicts the set $\mathcal{A}$ of all ordered pairs $(\delta,\gamma)$ that satisfy condition \eqref{Condition-delta-gamma} in Theorem \ref{theorem}. Figures \ref{fig:Gamma-} and \ref{fig:Gamma+} show plots of the modified error function $\Phi_{\delta\gamma}$ for several choices of the parameters $\delta$ and $\gamma$, obtained by solving numerically the boundary value problem \eqref{P} through the bvodes routine implemented in Scilab \cite{Ascher}. Note that condition \eqref{Condition-delta-gamma} is sufficient but not necessary, therefore we show the plots of the modified error function for parameters included and not included in the set $\mathcal{A}$.

\begin{figure}[h!]
\centering
\begin{subfigure}{.35\textwidth}
\includegraphics[scale=0.38,clip=true,trim=0 0 0 20]{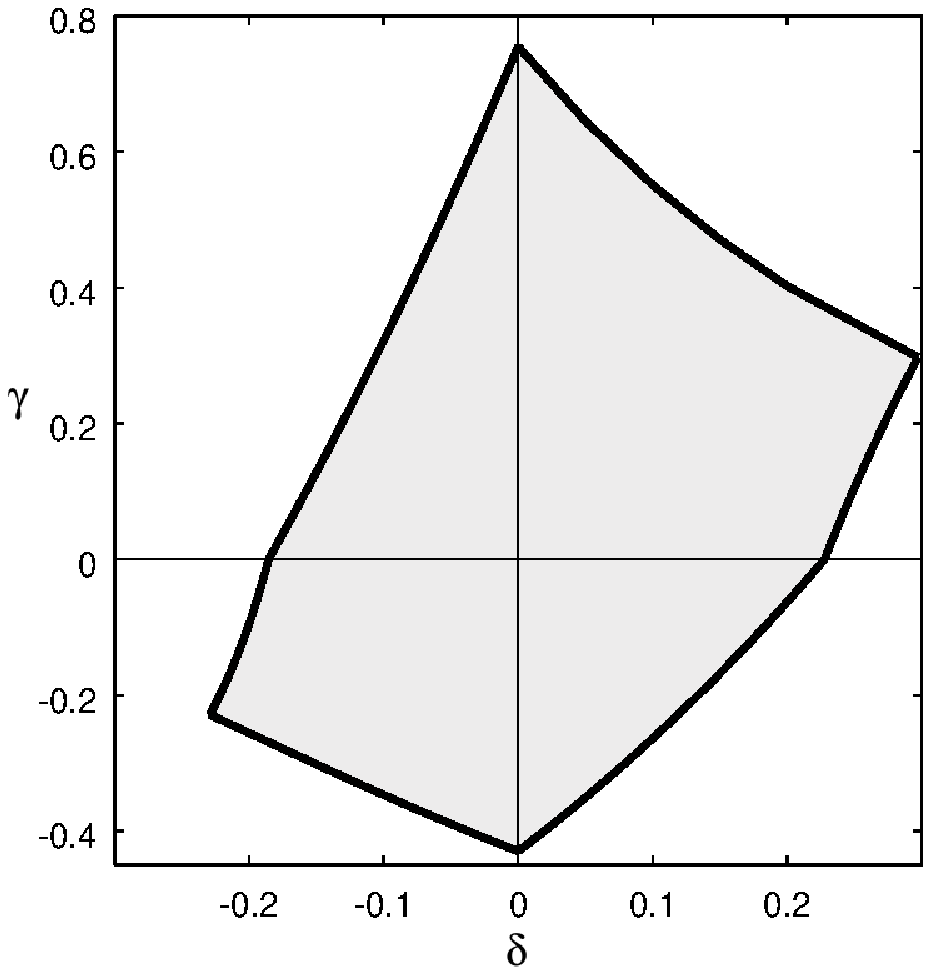}
\caption{$\mathcal{A}=\{(\delta,\gamma) : M(\delta,\gamma)<1\}$}\label{fig:A}
\end{subfigure}
\begin{subfigure}{.3\textwidth}
\includegraphics[scale=0.33,clip=true,trim=0 0 0 20]{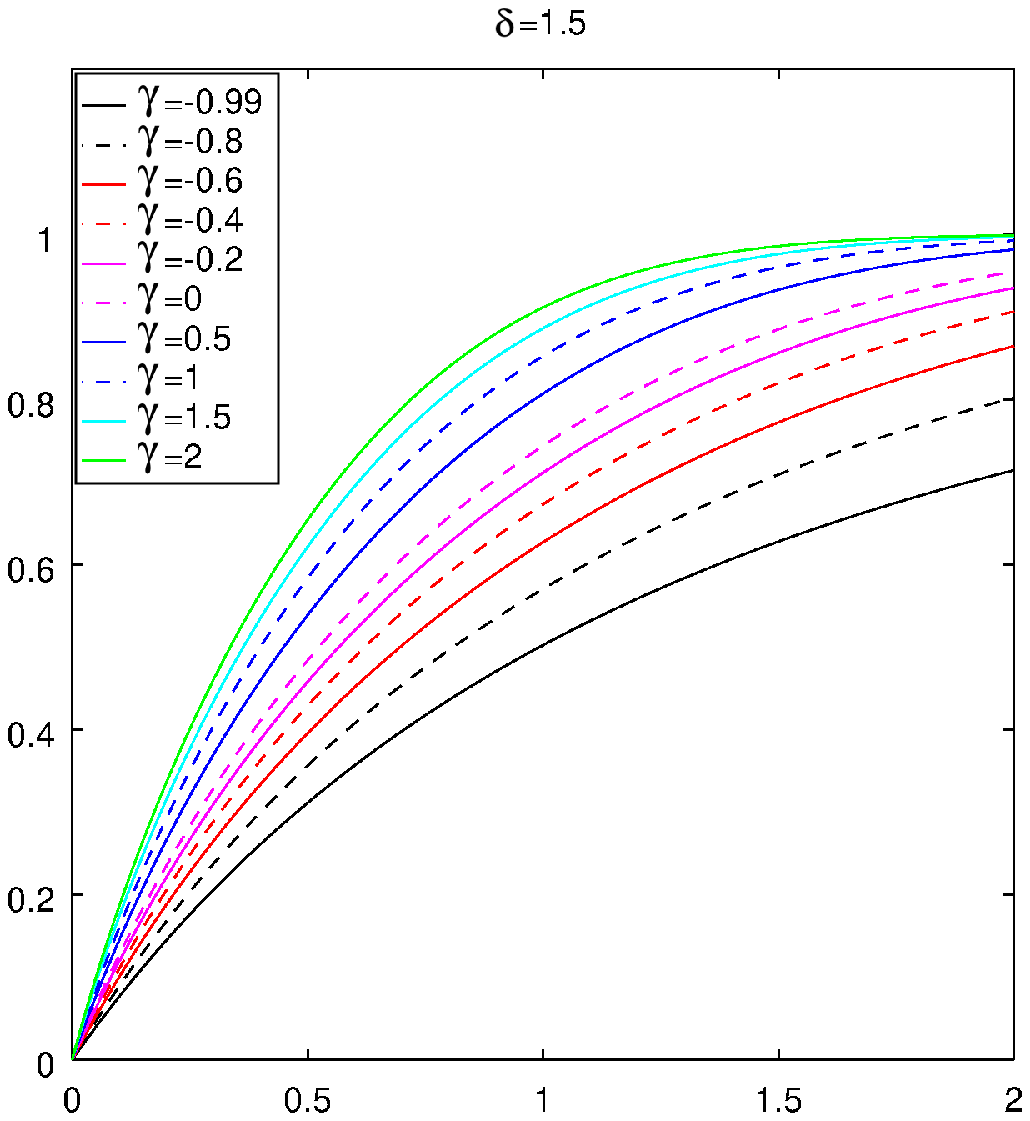}
\caption{$\Phi_{\delta\gamma}$ for $\delta=1.5$}\label{fig:Gamma-}
\end{subfigure}
\begin{subfigure}{.3\textwidth}
\includegraphics[scale=0.33,clip=true,trim=0 0 0 20]{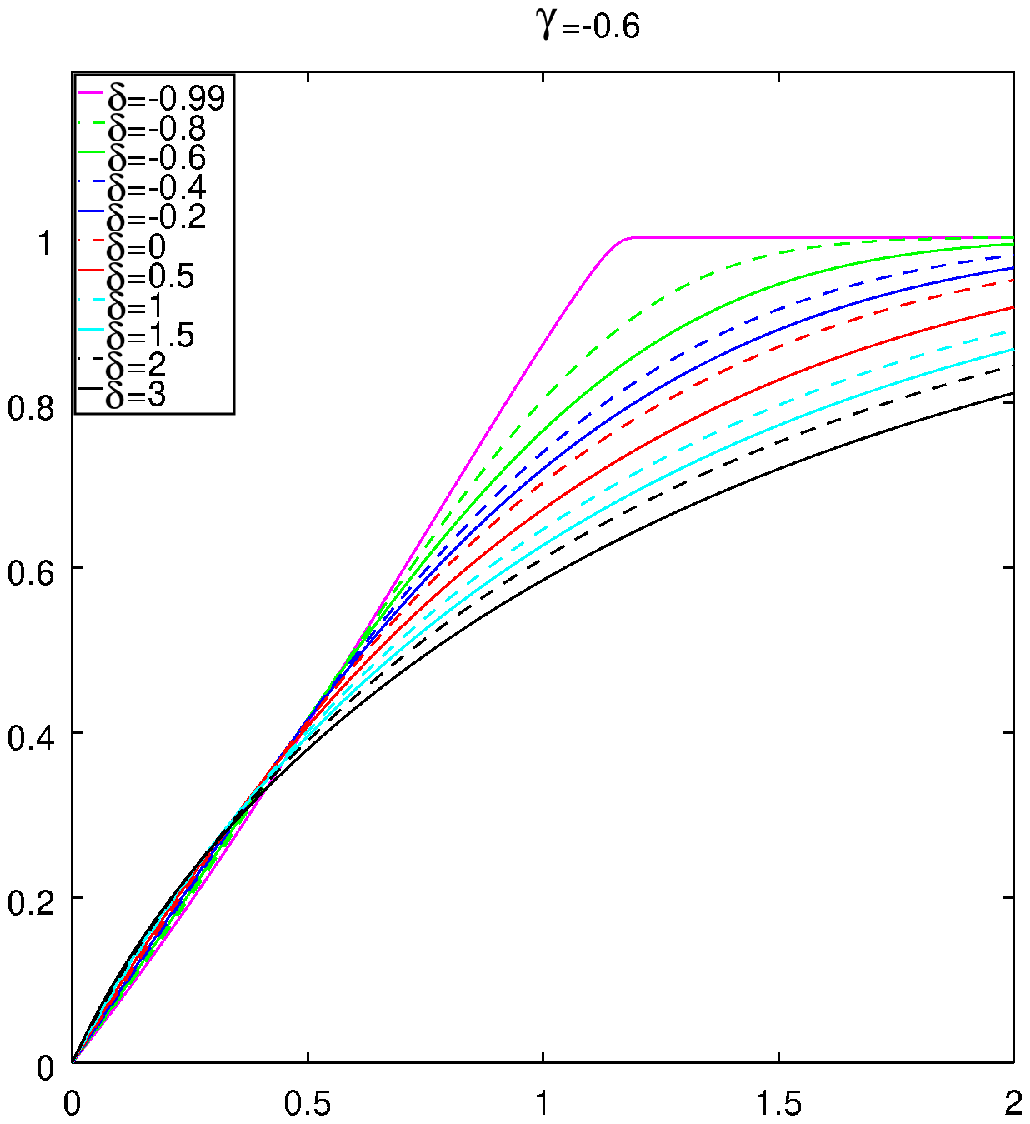}
\caption{$\Phi_{\delta\gamma}$ for $\gamma=-0.6$}\label{fig:Gamma+}
\end{subfigure}

\caption{Pairs $(\delta,\gamma)$ that satisfy condition \eqref{Condition-delta-gamma} in Theorem \ref{theorem} (left) and plots of the modified error function $\Phi_{\delta\gamma}$ for $\delta=1.5$  (middle) and $\gamma=-0.6$ (right).}
\end{figure}

The next corollary closes the article and establishes that the modified error function $\Phi_{\delta\gamma}$ given by Theorem \ref{theorem} shares some essential features with the classical error function. The proof is analogous to the one for Theorem 5.1 of \cite{CeSaTa2018} so that we shall omit it here.

\begin{corollary}\label{corollary}
Suppose that \eqref{Condition-delta-gamma} holds true. Then the unique solution $\Phi_{\delta\gamma}$ to problem \eqref{P} given by Theorem \ref{theorem} is increasing. If in addition $\delta$ is non-negative, then $\Phi_{\delta\gamma}$ is also concave.
\end{corollary}

\subsection*{Acknowledgments}
D. A. Tarzia and N. N. Salva have been supported by CONICET within the project PIP 0275, Universidad Austral, Rosario, Argentina and by the European Union's Horizon 2020 Research and Innovation Programme within the Marie Sklodowska-Curie grant agreement 823731 CONMECH. 
\bibliographystyle{unsrt} 
\bibliography{References}

\end{document}